\documentclass[12pt]{article}
\usepackage{amssymb, amsmath, amsthm, bm, float, graphicx,  tabularx, geometry, mathrsfs, setspace, enumerate, appendix, csquotes, url, yfonts, pslatex, sectsty, eulervm, xcolor, hyperref}
\usepackage[T1]{fontenc}
\usepackage[round]{natbib}
\usepackage{pgf}
\usepackage{tikz}
\usepackage[round]{natbib}

\usetikzlibrary{arrows, automata, positioning}

\newtheorem{definition}{Definition}

\newtheorem{proposition}{Proposition}
\newtheorem{remark}{Remark}

\newcommand{\RR}{\mathbb{R}}

\newcommand{\N}{\mathbb{N}}
\input{tcilatex}
\begin{document}

\title{On the Representation and Construction of Equitable Social Welfare
Orders\footnote{We sincerely thank the referee and the associate editor of this journal for several useful comments and constructive suggestions.}}
\author{Ram Sewak Dubey\thanks{%
Department of Economics, Feliciano School of Business, Montclair State
University, Montclair, NJ 07043; E-mail: dubeyr@montclair.edu} \and %
Giorgio Laguzzi\thanks{%
University of Freiburg in the Mathematical Logic Group at Eckerstr. 1, 79104
Freiburg im Breisgau, Germany; Email: giorgio.laguzzi@libero.it} \and %
Francesco Ruscitti\thanks{%
Department of Economics and Social Sciences, John Cabot University, Via
della Lungara 233, 00165 Rome, Italy. Email: fruscitti@johncabot.edu}}
\date{\today}
\maketitle

\begin{abstract}
This paper examines the representation and explicit description of social
welfare orders on infinite utility streams. It is assumed that the social welfare orders under
investigation satisfy upper asymptotic Pareto and anonymity axioms. We prove
that there exists no real-valued representation of such social welfare
orders. 
In addition, we establish that the existence of a social welfare order
satisfying the anonymity and upper asymptotic Pareto axioms implies the
existence of a non-Ramsey set, which is a non-constructive object. Thus, we
conclude that the social welfare orders under study do not admit explicit
description.

\noindent \emph{Keywords:} \texttt{Anonymity,}\; \texttt{Non-Ramsey set,}\; 
\texttt{Social Welfare Function,}\; \texttt{Social Welfare Order,}\; \texttt{%
Upper Asymptotic Pareto.}

\noindent \emph{Journal of Economic Literature} Classification Numbers: 
\texttt{C65,}, \texttt{D63,}\; \texttt{D71.}
\end{abstract}

\setcounter{MaxMatrixCols}{10}

\newpage

\section{Introduction}

\label{sec:1} This paper deals with efficiency and intergenerational equity
in the setting of policies that affect present and future generations.
Relevant questions are: how should a social planner weigh the welfare of the
present generation against the well-being of future generations? Is there a
conflict (and in what sense) between intergenerational equity and efficiency
in the evaluation of infinite utility streams? 
This subject has received wide
attention in the economics, philosophy and political science literature in
recent years. In this paper we investigate preference relations, on the
space of infinite utility streams, that are complete, transitive, invariant
to finite permutations, 
and respect some version
of the Pareto ordering: equitable preferences, for short. We stick to the
standard framework which concerns the problem of defining a social welfare
order on the set $X$ of infinite utility streams, where $X$ is of the form $%
X=Y^{\mathbb{N}}$, $Y$ denotes a non-empty subset of real numbers, and $%
\mathbb{N}$ is the set of natural numbers. There is a vast body of
literature on the subject matter. In what follows we will briefly overview
it in order to highlight and put our own contribution in context.

In a pioneering paper, \citet{ramsey1928} observed that discounting one
generation's utility relative to another's is 
\enquote{ethically
indefensible}, and something that 
\enquote{arises merely
from the weakness of the imagination}. Following in Ramsey's footsteps, %
\citet{diamond1965} introduced the concept of \emph{anonymity} (as an axiom
imposed on preferences over infinite utility streams) to formalize the
principle of equitable preferences (\enquote{equal
treatment} of present and future generations). This axiom requires that two
infinite utility streams be indifferent if one is obtained from the other by
interchanging the utility level of any two generations. There is also broad
consensus among scholars on another desirable attribute that preferences
should possess, namely the Pareto criteria. In its strongest form,
the Pareto principle asserts that one utility stream must be deemed strictly
better than another if at least one generation is better off and no
generation is worse off. Therefore, a question that naturally arises is
whether one can aggregate infinite utility streams with a social welfare
function,\footnote{%
A real-valued function representing a given social welfare order is referred
to as a social welfare function.} and consistently evaluate them while respecting anonymity and some form of the Pareto
axiom. This question was first approached formally by \citet{diamond1965}
who showed that, if the possible range of utilities in each period is the
closed interval $[0,1]$, a social welfare order that displays anonymity and
the strong Pareto ordering cannot be continuous in the topology induced by
the supremum norm. Hence, there does not exist any continuous (in the
topology induced by the sup norm) social welfare function satisfying the
anonymity and strong Pareto axioms. \citet{basu2003} refined Diamond's
result by showing that the non-representability result still holds when
continuity is dispensed with, and even for subsets $Y$ of the real numbers
containing only two elements. The bottom line is that when $Y$ contains at
least two elements, there exists no representable social welfare
order satisfying the anonymity and strong Pareto axioms. Hence, one cannot
exploit canonical constrained-maximization techniques to figure out an
optimal policy. One potential way out consists in weakening the strong
Pareto condition while still demanding that the social welfare order be
representable. However, \citet{crespo2009} established that if $Y$ contains
at least two elements, there is no social welfare function satisfying
anonymity and the infinite Pareto axiom.\footnote{%
According to the infinite Pareto axiom, one utility stream is strictly
better than another if infinitely many generations are better off and no
generation is worse off. So, infinite Pareto is weaker than strong Pareto.}
On the other hand, \citet{basu2007a} provided an example of a social welfare
function that satisfies anonymity and the weak Pareto principle.\footnote{%
The weak Pareto principle states that an infinite utility stream, say $x$, is
preferred to another, say $y$, if every generation is better off in $x$ than
in $y$. So, infinite Pareto is stronger than weak Pareto.} 
\citet{petri2019} examined a version of the Pareto axiom, namely lower
asymptotic Pareto, which is weaker than infinite Pareto and stronger than
weak Pareto. In a nutshell, given any pair $x$ and $y$ of infinite utility
streams, if $x$ dominates $y$ and the lower asymptotic density of the subset
of natural numbers such that $x_{n}>y_{n}$ is positive, the lower asymptotic
Pareto ordering requires $x$ to be preferred to $y$. 
Petri exhibited an explicit formula for a social welfare function satisfying lower asymptotic Pareto and anonymity under the assumption that Y is finite.
This result leaves us wondering whether there exists a social welfare function if one considers
a Pareto ordering that is weaker than infinite Pareto but stronger than
lower asymptotic Pareto. 
We address this issue by focusing on a version
of the Pareto ordering that we term upper asymptotic Pareto as it hinges on
the upper asymptotic density of the subset of natural numbers over which a
welfare improvement occurs. It is easy to see that upper asymptotic Pareto is
weaker than infinite Pareto. Moreover, as will become clear later, upper
asymptotic Pareto is stronger than lower asymptotic Pareto. 
Therefore, a question arises: is the existence of a social welfare function still guaranteed if $Y$ is any
non-trivial domain and one postulates upper asymptotic Pareto together with
anonymity? Proposition \ref{P1} below provides a negative answer to the
preceding question.\footnote{As a matter of fact, in the proof of Proposition \ref{P1} we use the concept of weak upper asymptotic Pareto (see definitions \ref{D1} through \ref{D4} below) which is weaker than upper asymptotic Pareto. Arguably, this makes our impossibility result more compelling.}

\citet{petri2019} found a social welfare function satisfying lower
asymptotic Pareto and anonymity on a domain $Y$ containing finitely many
elements, but we know from Proposition \ref{P1} below that there is no
numerical representation of a social welfare order satisfying weak upper
asymptotic Pareto and anonymity. Although a real-valued representation of an
underlying social welfare order can be
very useful, yet pairwise ranking of utility streams would suffice for the
purpose of policy-making as long as the binary relation at hand exists and
can be operationalized. Therefore, we wish to know if it is possible to describe explicitly (for the purpose of economic policy) a social
welfare order satisfying weak upper asymptotic Pareto and anonymity. To
provide some background on this line of inquiry, before we preview our own result, recall that \citet{svensson1980} established the existence of
a social welfare order that satisfies the anonymity and strong Pareto
axioms, assuming the set $Y$ of possible range of utilities to be the closed
interval $\left[ 0,1\right] $. However, his possibility result relies on
Szpilrajn's Lemma whose proof depends on the axiom of choice. Consequently,
this social welfare order cannot be used by policy makers for social
decision-making. In the wake of Svensson's result, \citet{fleurbaey2003}
conjectured that 
\enquote{there exists no
explicit description (that is, avoiding the axiom of choice or similar
contrivances) of an ordering which satisfies the Anonymity and Weak Pareto
axioms}. As shown by \citet{lauwers2010} and \citet{zame2007}, it turns out
that the axiom of choice is unavoidable for the existence of a
social welfare order satisfying the anonymity and Pareto axioms. The proof
of their result relies on the existence of non-Ramsey sets and
non-measurable sets, respectively. 
Similar to the above findings, in Proposition \ref{P2} of the present paper we show that the existence of a social welfare order satisfying anonymity and weak upper asymptotic Pareto (we know such order does exist, in view of  \citet{svensson1980}), on a domain $Y$ containing at least two elements, entails the existence of a non-Ramsey set.

In order to highlight the scope of our results within the existing literature, it is worth considering the following table. 
It summarizes some results on the representation and constructive nature of anonymous social welfare orders satisfying various forms of the Pareto axiom.%
\footnote{In the table below $|Y|$ denotes the cardinality of the set $Y$.} 
We defer further discussion on the question mark appearing in Table 1 to the concluding remarks.

\begin{table}[ht]
\caption{}
\centering 
\begin{tabular}{l l l l}
\hline\hline 
Pareto axiom&$|Y|$& Representation & Constructive nature \\ [0.25ex] 
\hline
Strong &$\geq 2$&No (\citet{basu2003}) & No (\citet{lauwers2010}) \\
&&&\quad\quad\citet{zame2007}\\
\hline
Infinite &$\geq 2$&No (\citet{crespo2009}) & No (\citet{lauwers2010}) \\ 
\hline
Upper Asymptotic &$\geq 2$& No (Proposition \ref{P1})  & No (Proposition \ref{P2})  \\ 
\hline
Lower Asymptotic &Finite& Yes  (\citet{petri2019}) & Yes  (\citet{petri2019})\\ 
\hline
Lower Asymptotic &Infinite& No (\citet{petri2019}) & \qquad ?\\ 
\hline\hline
\end{tabular}
\end{table}

The remainder of the paper is organized as follows. 
In section \ref{sec:2} we introduce the basic notation which will be used throughout the paper and gather all the definitions.
In section \ref{sec:3} we state and prove our main
results (Propositions \ref{P1} and \ref{P2}). Section \ref{sec:4} concludes.

\section{Preliminaries}

\label{sec:2}

Let $\mathbb{R}$, $\mathbb{Q}$, and $\mathbb{N}$ be the set of real numbers,
rational numbers, and natural numbers, respectively. For $y$, $z\,\in \,%
\mathbb{R}^{\mathbb{N}}$, we write $y\geq z$ if $y_{n}\geq z_{n}$, for all $%
n\in \mathbb{N}$; $y>z$ if $y\geq z$ and $y\neq z$; and $y\gg z$ if $%
y_{n}>z_{n}$ for all $n\in \mathbb{N}$.

\subsection{Social Welfare Orders}

\label{s2.2}

Let $Y\subset \mathbb{R}$ be the set of all possible utilities that any
generation can achieve. Then, $X\equiv Y^{\mathbb{N}}$ is the set of all
feasible utility streams. We denote an element of $X$ by $x$, or, alternately
by $\langle x_{n}\rangle $, depending on the context. 
If $\langle x_{n}\rangle \in X$, then $\langle x_{n}\rangle = \left(x_{1}, x_{2},\cdots\right)$, where $x_{n}\in Y$ represents the amount of utility earned by the $n^{\text{th}}$ generation.

A binary relation on $X$ is denoted by $\succsim $. Its symmetric and
asymmetric parts, denoted by $\sim$ and $\succ$, respectively, are defined
in the usual way. A social welfare order (SWO henceforth) is by definition a
complete and transitive binary relation. Given a SWO $\succsim $ on $X$, one
says that $\succsim $ can be represented by a real-valued function,
called a social welfare function (SWF henceforth), if there is a mapping $%
W:X\rightarrow \mathbb{R}$ such that for all $x$, $y\in X$, $x\succsim y$ if
and only if $W(x)\geq W(y)$.

It is useful to recall the definitions of lower and upper asymptotic density of a set $S\subset 
\mathbb{N}$. As usual, we will let $|\cdot |$ denote the cardinality of a given
finite set. The lower asymptotic density of $S$ is defined as follows: 
\begin{equation*}
\underline{d}(S)=\underset{n\rightarrow \infty }{\liminf }\;\frac{|S\cap
\{1,2,\cdots ,n\}|}{n}.
\end{equation*}%
Similarly, the upper asymptotic density of $S$ is defined as follows: 
\begin{equation*}
\overline{d}(S)=\underset{n\rightarrow \infty }{\limsup }\;\frac{|S\cap
\{1,2,\cdots ,n\}|}{n}.
\end{equation*}

\subsection{Equity and Efficiency Axioms}

\label{s2.3}

We will be dealing with the following equity and efficiency axioms that we
may want the SWO to satisfy.

\begin{definition}
\label{D1} \emph{Anonymity (AN henceforth): If $x, y\in X$, and there exist $i, j\in 
\mathbb{N}$ such that $y_{j} = x_{i}$ and $x_{j} = y_{i}$, while $y_{k} =
x_{k}$ for all $k\in \mathbb{N}\setminus \{i, j\}$, then $x\sim y$.}
\end{definition}

\begin{definition}
\label{D2} \emph{Upper Asymptotic Pareto (UAP henceforth): Given $x, y\in X$, if $%
x\geq y$ and $x_{i}>y_{i}$ for all $i$}$\in $\emph{$S\subset \mathbb{N}$
with $\overline{d}(S)>0$, then $x\succ y$.}
\end{definition}

\begin{definition}
\label{D3} \emph{Weak Upper Asymptotic Pareto (WUAP henceforth): Given $x, y\in X$,
if $x\geq y$ and $x_{i}>y_{i}$ for all $i$}$\in $$S\subset \mathbb{N}$ \emph{%
with $\overline{d}(S)=1$, then $x\succ y$.}
\end{definition}

\begin{definition}
\label{D4} \emph{Lower Asymptotic Pareto (LAP henceforth): Given $x, y\in X$, if $%
x\geq y$ and $x_{i}>y_{i}$ for all $i$}$\in $$S\subset \mathbb{N}$ \emph{%
with $\underline{d}(S)>0$, then $x\succ y$.}
\end{definition}

Of course, WUAP is weaker than UAP. 
Moreover, UAP is stronger than (and different from) LAP.
To see why this is the case, we refer the reader to Remark \ref{R1}.

\subsection{Non-Ramsey collection of sets}

\label{s2.4}

Let $T$ be an infinite subset of $\mathbb{N}$. We denote by $\Omega (T)$ the
collection of all infinite subsets of $T$, and we will refer to $\Omega (\mathbb{N})$ simply as $\Omega $. Thus, any infinite subset $T$ of $\mathbb{N}$ belongs to $%
\Omega$. A collection of sets $\Gamma \subset \Omega $ is called Ramsey if there exists $T\in \Omega $ such that either $\Omega (T)\subset \Gamma $
or $\Omega (T)\subset \Omega \diagdown \Gamma $. 
We can next define a collection of sets known as non-Ramsey.

\begin{definition}
\label{D5} \emph{A collection of sets $\Gamma \subset \Omega $ is said to be
non-Ramsey if for every $T\in \Omega $, the collection $\Omega (T)$
intersects both $\Gamma $ and its complement $\Omega \diagdown \Gamma $.}
\end{definition}

We refer the reader to \citet{fleurbaey2003}, \citet[Section 4]{zame2007}, %
\citet[Section 4]{lauwers2010}, \citet[Section 2.2.5]{dubey2014b}, %
\citet[Sections 2 and 3]{laguzzi2020} and \citet[Section 5]{dubey2019} for
a detailed account of the relevance of non-constructive objects (e.g., non-Ramsey
sets, non-measurable sets, non-Baire sets etc.) to economics.

\section{Results}

\label{sec:3} In this section we state and prove the main results of this
paper. Define $f:\mathbb{N}\rightarrow \mathbb{N}$ by 
\begin{equation}
f(1):=1,\;\text{and}\;f(n+1):=(n+1)f(n)=(n+1)!\;\text{for all}\;n>1.
\label{part1}
\end{equation}%
Next we use (\ref{part1}) to construct the following partition of $\mathbb{N}$ which will play an important role in
the proof of Propositions \ref{P1} and \ref{P2} below. 
\begin{equation}
I_{1}:=f(1)=\{1\},\;\text{and}\;I_{n}:=\left(f(n-1), f(n)\right] \cap 
\mathbb{N}\;\text{for}\;n\geq 2\text{.}  \label{part2}
\end{equation}%
Note that $|I_1|=1$, and 
\begin{equation*}
|I_{n}|=f(n)-f(n-1)=nf(n-1)-f(n-1)=(n-1)f(n-1),\;\text{for}\; n>1,\;\text{and%
}
\end{equation*}%
\begin{equation*}
\sum_{m=1}^{n}|I_{m}|=f(1)+\sum_{m=2}^{n}\left[ f(m)-f(m-1)\right] =f(n).
\end{equation*}%
Also, note that 
\begin{equation}
\alpha _{n}:=\frac{|I_{n}|}{\sum_{m=1}^{n}|I_{m}|}=\frac{(n-1)f(n-1)}{f(n)}=%
\frac{(n-1)f(n-1)}{nf(n-1)}=\frac{n-1}{n}=1-\frac{1}{n}.  \label{eq-density}
\end{equation}%
Observe that $\alpha _{n}\rightarrow 1$ as $n\rightarrow \infty $.%

\subsection{No social welfare function satisfies upper asymptotic Pareto and anonymity}

We first prove that there is no social welfare function satisfying UAP and AN. 
We exploit techniques used in \citet{basu2003} and \citet{crespo2009} together with the partition of the set of natural numbers introduced above (see (\ref{part2})).

\begin{proposition}\label{P1} 
There does not exist any social welfare function satisfying UAP and AN on $X=Y^{\mathbb{N}}$, with $Y=\{a, b\}$ and $a<b$.
\end{proposition}

\begin{proof}
We establish the claim by contradiction.
In the following proof we employ WUAP instead of UAP in order to stress that our result is robust to a weaker specification of the Pareto axiom.
Let $W: X\rightarrow \RR$ be a SWF satisfying WUAP and AN.
We let $a=0$ and $b=1$ without any loss of generality.
Let $q_1$, $q_2, \cdots$ be an enumeration of rational numbers in $[0, 1]$.
We keep this  enumeration fixed throughout the proof.
Let $r\in (0, 1)$.
Based on the above enumeration of rational numbers, we construct a sequence $x(r)$ as detailed below.
Let $l_1(r) = \min \left\{ n\in \N: q_n \in (0, r)\right\}$.
Having defined $l_1(r)$, for every $k\geq 1$ we set
\[
l_{k+1}(r)= \min \left\{ n\in \N \setminus \{l_1(r), l_2(r), \cdots, l_k(r)\}: q_n \in (0, r)\right\}.
\]
Note that $l_1(r)<l_2(r)<\cdots<l_{k}(r)<\cdots$.
Thus, we can define $L(r)$ as follows:
\[
L(r) = \{l_1(r), l_2(r), \cdots, l_k(r), \cdots\}.
\]
Now, let $U(r) = \{u_1(r), u_2(r), \cdots, u_k(r), \cdots\}$ denote the set $\N\setminus L(r)$, with
\[
u_1(r)<u_2(r)< \cdots< u_k(r)<u_{k+1}(r)<\cdots.
\]
\noindent We are ready to define the utility stream  $\langle x(r)\rangle$ as follows:%
\footnote{It is defined in blocks of $f(n)-f(n-1)$ elements at a time (for $n\geq 2$).
The initial $|I_{1}|$ element of $\langle x(r)\rangle$ equals $1$ if $1\in L(r)$, and $0$ otherwise.
The next $|I_{2}|$ element of $\langle x(r)\rangle$ equal $1$ if $2\in L(r)$, and $0$ otherwise, and so on.
Observe that $\langle x(r)\rangle$ is well-defined.}
\begin{equation}\label{03}
x_n(r) = \left\{ 
\begin{array}{ll}
1 & \text{if}\; n\in I_{l}\; \text{and}\; l\in L(r),\\
0 & \text{otherwise}.%
\end{array}
\right.
\end{equation}

\noindent Next, we select from the (fixed) enumeration of rational numbers a strictly decreasing sequence  $\langle q_{n_k(r)}\rangle\in (r, 1)$ which is convergent to $r$. 
Observe that the sequence $\left\{n_k(r): k\in\N\right\}$ is a sub-sequence of $\left\{u_n(r): n\in\N\right\}$.
Let $\Delta(r):=$  $\underset{k\in \N}{\cup} I_{n_k(r)}$.
We define another utility stream $\langle z(r)\rangle$ as follows:

\begin{equation}\label{04}
z_n(r) = \left\{ 
\begin{array}{ll}
1 & \text{if}\; n\in \Delta(r),\\
x_n(r) & \text{otherwise}.%
\end{array}
\right.
\end{equation}
Note that for every $n\in \Delta(r)$, $z_n(r)=1>0=x_n(r)$, therefore $z_n(r)\geq x_n(r)$ for every $n\in \N$.
Observe that for each term in the sequence $\left\{n_k(r): k\in \N\right\}$,%
\footnote{In the remainder of the proof we omit reference to $(r)$ for ease of notation, whenever no ambiguity arises from the context.}
using (\ref{eq-density}) above, we have
\begin{align*}
\alpha_{n_k} &= \frac{|I_{n_k}|}{\sum_{m=1}^{n_k}|I_{m}|}=1- \frac{1}{n_k}.
\end{align*}
Also, notice that $\Delta(r)\cap \{1, 2, \cdots, f(n_k)\}\supset I_{n_k}$ for each $k\in\N$.
Therefore, since $\alpha_{n_k}\rightarrow 1$ as $k\rightarrow \infty$, we get $\overline{d} (\Delta(r)) =1$.
Hence, by WUAP $x(r)\prec z(r)$, therefore
\begin{equation}\label{P1Ea}
 W(x(r))<W(z(r)).
\end{equation}
Next, we pick $s\in (r, 1)$.
To such an $s$ there correspond the sequences $\langle x(s)\rangle$ and $\langle z(s) \rangle$ according to (\ref{03}) and (\ref{04}), respectively.
In order to rank $z(r)$ and $x(s)$, we need to consider the following two possibilities.
\begin{enumerate}[(a)]
\item{$q_{n_1}<s$. 
In this case, for any $n \in \N$, if $z_n(r) =1$ then by construction of $x(s)$ we must have $x_n(s)=1$ as well.
Therefore, $x_n(s)\geq z_n(r)$ holds true for all $n\in\N$.
Let $\Delta (rs):=\underset{k^{\prime}\in \N}{\bigcup} I_{v_{k^{\prime}}}$, where $v_{k^{\prime}}\in (U(r)\cap L(s)) \setminus \{n_1, n_{2}, \cdots\}$.
Observe that there are infinitely many $q_{v_{k^{\prime}}}$ in the interval $\left[r,s\right)\setminus \left\{q_{n_k(r)}, k\in \N\right\}$.
Then, $z_n(r) = 0 < 1 = x_n(s)$ for every $n\in \Delta (rs)$.
By (\ref{eq-density}) above, let 
\[
\alpha_{v_{k^{\prime}}}= \frac{|I_{v_{k^{\prime}}}|}{\sum_{m=1}^{v_{k^{\prime}}}|I_{m}|} = 1- \frac{1}{v_{k^{\prime}}}.
\]
Observe that $v_{k^{\prime}} \rightarrow \infty$ as $k^{\prime} \rightarrow \infty$, therefore $\alpha_{v_{k^{\prime}}} \rightarrow 1$.
By WUAP, $z(r)\prec x(s)$, consequently \begin{equation}\label{P11Ec}
W(z(r)) < W(x(s)).
\end{equation}}
\item{$q_{n_1}\geq s$. 
First we observe that $q_{n_k}<s$ for all but finitely many $n_k$.
Hence, we can pick $K$, $K$ being finite, such that $q_{n_1}\geq s$, $\cdots$, $q_{n_K}\geq s$.
Then, for every $n$ belonging to $I_{n_1}$, $I_{n_2}$, $\cdots$, $I_{n_K}$ (there are finitely many such $n$), we have $z_n(r) =1>0=x_n(s)$.
There exist infinitely many $l_m(s) \in \N \setminus \{n_1, n_2, \cdots, n_K\}$, with $l_m(s)>n_K$, that are distinct from the subsequence $\left\{n_k\right\}$ and are such that $q_{l_m(s)}\in [r, q_{n_K})\cap [r, s)$.
For every $l_m(s)$ there are $|I_{l_m(s)}|$ elements of the utility stream $\langle x(s)\rangle$ such that $x_n(s)=1>0=z_n(r)$.
We interchange the $I_{n_1}$, $\cdots$, $I_{n_K}$ coordinates of $\langle z(r)\rangle$ (having value $1$) with an equal number of elements (having value $0$) from the $I_{l_m(s)}$, $I_{l_m^{\prime}(s)}$, $\cdots$ so as to obtain the utility stream $\langle z^{\prime}\rangle$.
It follows from AN that $z^{\prime}\sim z(r)$, hence
\begin{equation}\label{P1Eb}
W(z^{\prime}) = W(z(r)).
\end{equation}
Compare $\langle z^{\prime}\rangle$ to $\langle x(s)\rangle$, and observe that $z^{\prime}_n = 0 = x_n(s)$ for every $k$ and $n\in I_{u_k(s)}$.
Also, $z^{\prime}_n = 1 = x_n(s)$ for every  $k$ and $n\in I_{l_k(r)}$.
Moreover, $z^{\prime}_n = 1 = x_n(s)$ for every $k>K$ and $n\in I_{n_k}$.
Let $\Delta^{\prime} (rs):=\underset{k^{\prime}\in \N}{\cup} I_{v_{k^{\prime}}}$ where $v_{k^{\prime}}\in (U(r)\cap L(s)) \setminus \{n_K, n_{K+1}, \cdots\}$.
Thus, $z_n^{\prime} = 0 < 1 = x_n(s)$ for every $n\in \Delta^{\prime}(rs)$.
By (\ref{eq-density}) above, let 
\[
\alpha_{v_{k^{\prime}}}= \frac{|I_{v_{k^{\prime}}}|}{\sum_{m=1}^{v_k^{\prime}}|I_{m}|} = 1- \frac{1}{v_{k^{\prime}}}.
\]
Observe that $v_{k^{\prime}} \rightarrow \infty$ as $k^{\prime} \rightarrow \infty$, therefore $\alpha_{v_{k^{\prime}}} \rightarrow 1$.
By WUAP, $z^{\prime}\prec x(s)$, consequently \begin{equation}\label{P1Ec}
W(z^{\prime}) < W(x(s)).
\end{equation}}
\end{enumerate}
Therefore,  because the two cases are mutually exclusive and exhaustive, (\ref{P1Ea}) and (\ref{P11Ec}) and (\ref{P1Ea}), (\ref{P1Eb}), and (\ref{P1Ec}) imply that $(W(x(r)), W(z(r))$ and $(W(x(s)), W(z(s))$ are non-empty and disjoint open intervals.    
Hence, because $r$ and $s$, with $r<s$, were arbitrary, by density of $\mathbb{Q}$ in $\mathbb{R}$ we conclude that we have found a one-to-one mapping from $(0, 1)$ to $\mathbb{Q}$, which is impossible as the latter set is countable.
\end{proof}

\subsection{The non-constructive nature of a social welfare order satisfying upper asymptotic Pareto
and anonymity}

Proposition \ref{P1} leaves us wondering if one can describe the SWO under consideration despite the latter having no real-valued representation. 
Therefore, in Proposition \ref{P2} below we restrict ourselves to the same setting as in Proposition \ref{P1} and we show that no social welfare order satisfying anonymity and weak upper asymptotic Pareto is susceptible of an explicit description. To this end, it will
suffice to prove that the existence of such a social welfare order, when Y
contains only two elements, entails the existence of a non-Ramsey
set.
The proof of the following proposition is inspired by \citet{lauwers2010}.


\begin{proposition}
\label{P2} Let $Y=\{a,b\}$, with $a<b$, and assume that there is a social
welfare order on $X=Y^{\mathbb{N}}$ satisfying UAP and AN. Then, there
exists a non-Ramsey set.
\end{proposition}

\begin{proof}
We already know that UAP is stronger than WUAP. 
Thus, it will be enough to prove the claim assuming that the given SWO satisfies WUAP.
Given any $N:=\{n_1, n_2, \cdots n_k, \cdots\}$ (where $n_k<n_{k+1}$ for all $k\in\N$) that  belongs to $\Omega$, using (\ref{part1}) above we define recursively the following partition of natural numbers:
\[
I_1(N):= [0, f(n_1))\cap \N\; \text{and}\; I_{k}(N):= [f(n_{k-1}), f(n_{k}))\cap\N\; \text{for}\; k>1.
\]
Next, we define $x(N)$, $y(N) \in X$ as follows:
\begin{equation}\label{P2E1}
x_{t} (N)=
\begin{cases}
a \quad &\text{if $t \in I_k(N)$ and $k$ is odd}\\ 
b \quad &\text{if $t \in I_k(N)$ and $k$ is even},
\end{cases}
\end{equation}
\begin{equation}\label{P2E2}
y_{t} (N)=
\begin{cases}
a \quad &\text{if $t \in I_1(N)$,\; or\; $t\in I_k(N)$ and $k$ is even}\\ 
b \quad &\text{if $t \in I_k(N)$ and $k$ is odd, $k>1$.}
\end{cases}
\end{equation}
Let $\Gamma:= \{N \in \Omega: x(N) \prec y(N)\}$.
We claim that $\Gamma$ is a non-Ramsey set.
According to Definition \ref{D5}, we must show that for every $T \in \Omega$ there exists $S \in \Omega(T)$ such that $T \in \Gamma \Leftrightarrow S \notin \Gamma$. 
Pick any arbitrary $T:= \{t_1, t_2, \cdots, t_k, \cdots\}$, where $t_k<t_{k+1}$ for all $k\in\N$.
We distinguish three cases.
\begin{enumerate}[(1)]
\item $x(T) \prec y(T)$, therefore $T\in \Gamma$.
In this case, let $S:= T \setminus \{t_1, t_{4k+1}, t_{4k+2}: k\in \N \} = \{t_2, t_3, t_4, t_7, t_8, \cdots\}$.
Note that $I_1(S) = I_1(T)\cup I_2(T)$, $I_2(S) =I_3(T),\; I_{2k+1}(S) =I_{4k}(T)$, for all $k\geq 1$, and $I_{2k}(S) = I_{4k-3}(T)\cup I_{4k-2}(T)\cup I_{4k-1}(T)$ for all $k\geq 2$.
Therefore, by (\ref{P2E1}) and (\ref{P2E2}) we have $y_t(S) = a < b = x_t(T)$ for $ t\in I_2(T) \cup\; I_{4k+2}(T)$ and $y_t(S) =x_t(T)$ for all remaining $t\in\N$.
Define 
\[
\Delta:= \{t \in \N: y_t(S) < x_t(T)\}\; \text{and}\;
\delta_k:= \frac{|\Delta \cap [0, f(t_{4k+2}))|}{f(t_{4k+2})}.
\] 
Then, $I_2(T) \subset \Delta$, and $I_{4k+2}(T)\subset \Delta$ for all $k\in\N$.
For every $k \geq 1$ we have
\begin{equation}\label{C1E1}
\Delta \cap [0, f(t_{4k+2})) \supseteq I_{4k+2}(T)
\end{equation}
and $[f\left(t_{4k+2}-1\right), f\left(t_{4k+2}\right))\cap\N \subset I_{4k+2}(T)$.
Therefore, $|I_{4k+2}(T)|\geq f\left(t_{4k+2}\right) - f\left(t_{4k+2}-1\right) = \left(t_{4k+2}-1\right)f\left(t_{4k+2}-1\right)$, and
\begin{equation}\label{C1E2}
\sum_{j \leq 4k+2} |I_j(T)| = f\left(t_{4k+2}\right).
\end{equation}
Also,
\[
\frac{|I_{4k+2}(T)|}{\sum_{j \leq 4k+2} |I_j(T)|} \geq \frac{\left(t_{4k+2}-1\right)f\left(t_{4k+2}-1\right)}{f\left(t_{4k+2}\right)} = \frac{\left(t_{4k+2}-1\right)f\left(t_{4k+2}-1\right)}{\left(t_{4k+2}\right)f\left(t_{4k+2}-1\right)}=1- \frac{1}{t_{4k+2}}.
\]
Hence, by (\ref{C1E1}) and (\ref{C1E2}) and the above inequality, we have
\[
\delta_k:= \frac{|\Delta \cap [0,f(t_{4k+2}))|}{f(t_{4k+2})} \geq \frac{|I_{4k+2}(T)|}{\sum_{j \leq 4k+2} |I_j(T)|} \geq 1- \frac{1}{t_{4k+2}}.
\] 
Hence, $\overline{d}(\Delta) = 1$.
This is because given $\langle n_k: k \geq 1 \rangle$, with $n_k:=f(t_{4k+2})$, $\delta_k$ is a subsequence such that 
\[
\overline{d}(\Delta)=\limsup_{n \rightarrow \infty} \frac{|\Delta \cap \{1, \cdots, n\}|}{n} \geq \lim_{k \rightarrow \infty} \delta_k = 1.
\] 
Thus, we have found a set $\Delta \in \Omega$  such that $\overline{d} (\Delta) = 1$ and $y_t(S) = a < b = x_t(T)$ for $t \in \Delta$, and $y_t(S) =x_t(T)$ for all remaining $t\in\N$.
Therefore, it follows from WUAP that 
\begin{equation}\label{C1E3}
y(S) \prec x(T).
\end{equation} 
Since $y_t(T) = a < b = x_t(S)$ for $t\in I_{4k+2}(T)$, and $y_t(T) =x_t(S)$ for all remaining $t\in\N$, by the same logic one can prove that WUAP implies 
\begin{equation}\label{C1E4}
y(T) \prec x(S).
\end{equation}
Therefore, by (\ref{C1E3}) and (\ref{C1E4}) we get $y(S) \prec x(T) \prec y(T) \prec x(S)$.
By transitivity, $y(S) \prec x(S)$, which establishes that $S\notin \Gamma$, as was to be proven.

 
\item $y(T) \prec x(T)$, therefore $T\notin \Gamma$.
Let $S:= T \setminus \{t_1, t_{4k}, t_{4k+1}: k \in \N\} = \{t_2, t_3, t_6, t_7, t_{10}, \cdots \}$.
\noindent Note that $I_1(S) = I_1(T)\cup I_2(T)$, $ I_{2k+1}(S) = I_{4k}(T)\cup I_{4k+1}(T)\cup I_{4k+2}(T)$, and $I_{2k}(S) = I_{4k-1}(T)$, for all $k\in \N$.
Therefore, by (\ref{P2E1}) and (\ref{P2E2}) we have $x_t(S) = a < b = y_t(T)$ for $t\in I_{4k+1}(T)$, and $x_t(S) =  y_t(T)$ for all remaining $t\in \N$.
As in case (1) above, one can prove that  WUAP implies
\begin{equation}\label{C2E3}
x(S)\prec y(T).
\end{equation}

Furthermore, $x_t(T) = a < b = y_t(S)$ if $t\in I_{4k+1}(T)$, $y_t(S) = a < b = x_t(T)$ if $t\in I_{2}(T)$, and $x_t(S) =  y_t(T)$ for all remaining $t\in \N$.
Interchanging finitely many coordinates of $y(S)$ that lie in $I_2(T)$ with an equal number of coordinates in $I_5(T)$ yields the auxiliary sequence $y^{\prime}(S)$.
Hence, AN implies 
\begin{equation}\label{C2E4}
y^{\prime}(S)\sim y(S).
\end{equation}
Also, $x_t(T) = a < b = y^{\prime}_t(S)$ if $t\in I_{4k+1}(T)$, with $k\geq2$, and $x_t(S) =  y^{\prime}_t(S)$ for all remaining $t\in \N$.
As in case (1) above, using WUAP one can prove that
\begin{equation}\label{C2E5}
x(T)\prec y^{\prime}(S).
\end{equation}
Thus, it follows from (\ref{C2E4}), (\ref{C2E5}), and transitivity that
\begin{equation}\label{C2E6}
x(T)\prec y(S).
\end{equation} 
Therefore, by (\ref{C2E3}) and (\ref{C2E6}) we get  $x(S)\prec y(T)\prec x(T)\prec y(S)$.
By transitivity, $x(S)\prec y(S)$, which yields $S\in \Gamma$, as was to be proven.

 
\item $x(T) \sim y(T)$, therefore $T\notin \Gamma$.
Let $S:= T \setminus \{t_{4k-1}, t_{4k}: k \in \N\} = \{t_1, t_2, t_5, t_6, t_{9}, \cdots \}$.
Note that $I_1(S) = I_1(T)$, $I_{2k+1}(S) = I_{4k-1}(T)\cup I_{4k}(T)\cup I_{4k+1}(T)$, and
$I_{2k}(S) = I_{4k-2}(T)$ for all $k\in \N$.
Then, by (\ref{P2E1}) and (\ref{P2E2}) we have $x_t(S) = a < b = x_t(T)$ for $t\in I_{4k}(T)$, and  $x_t(S) =  x_t(T)$ for all remaining $t\in \N$.
As in case (1) above, WUAP implies
\begin{equation}\label{C3E1}
x(S) \prec x(T).
\end{equation}
Furthermore, $y_t(T) = a < b = y_t(S)$ for $t\in I_{4k}(T)$, and $y_t(T) =  y_t(S)$ for all remaining $t\in \N$.
As in case (1) above, one can prove that WUAP implies
\begin{equation}\label{C3E2}
y(T) \prec y(S).
\end{equation}
Therefore, by (\ref{C3E1}) and (\ref{C3E2}) $x(S)\prec x(T)\sim y(T)\prec y(S)$.
By transitivity, $x(S)\prec y(S)$.
Therefore, $S\in \Gamma$, as was to be proven.
\end{enumerate}
\end{proof}
It is worth pointing out that upper asymptotic Pareto is strictly stronger than lower asymptotic Pareto (see Remark \ref{R1} below). 
Consequently, the above propositions offer a novel result that can be contrasted with Petri's: while there exists a SWF satisfying anonymity and lower asymptotic Pareto if Y is finite (\citet{petri2019}), there is neither an explicit description (Proposition \ref{P2} above) nor a real-valued representation (Proposition \ref{P1} above) of a SWO that satisfies anonymity and upper asymptotic Pareto.

\begin{remark}\label{R1}
\emph{In what follows we substantiate our claim that upper asymptotic Pareto is indeed strictly stronger than lower asymptotic Pareto.
We accomplish this by showing first that the lower asymptotic density of the set $\Delta$ constructed in the proof of case (1) of Proposition \ref{P2}  is zero (while its upper asymptotic density is one, as we already know from that proof), and then by sketching the proof that the lower asymptotic density of $\Delta(r)$, $\Delta(rs)$ and $\Delta^{\prime}(rs)$ used in the proof of Proposition \ref{P1} is zero as well (recall that the upper asymptotic density of the foregoing sets is one).}

\emph{We know that $I_{2}(T)\subset \Delta $ and $I_{4k+2}(T)\subset \Delta $ for all $k\in \mathbb{N}$.
Observe that $|I_{2}(T)|=f(t_{2})-f(t_{1})<f(t_{2})$.}
\emph{Similarly, $|I_{4k-2}(T)|=f(t_{4k-2})-f(t_{4k-3})<f(t_{4k-2})$ for each $k\in \mathbb{N}$. 
Therefore, $\Delta \cap \lbrack 0,f(t_{4k-2})) = I_{2}(T)\cup I_{6}(T)\cup \cdots I_{4k-2}(T)$, and} 
\begin{equation}
|\Delta \cap \lbrack 0,f(t_{4k-2}))| = |I_{2}(T)\cup I_{6}(T)\cup \cdots I_{4k-2}(T)|<kf(t_{4k-2}).  \label{R1E1}
\end{equation}%
\emph{Also,} 
\begin{equation}
t_{4k+1}\geq t_{4k-2}+3.  \label{R1E2}
\end{equation}%
\emph{Moreover, notice that} $|\Delta \cap \lbrack 0,f(t_{4k-2}))|=|\Delta \cap
\lbrack 0,f(t_{4k+1}))|$. \emph{Therefore, it follows from (\ref{R1E1}) and (\ref{R1E2}%
) that} 
\begin{align*}
\frac{|\Delta \cap \lbrack 0,f(t_{4k+1}))|}{f(t_{4k+1})}&=
\frac{|\Delta \cap \lbrack 0,f(t_{4k-2}))|}{f(t_{4k+1})} \leq \frac{%
kf(t_{4k-2})}{f(t_{4k-2}+3)}\\
&=\frac{kf(t_{4k-2})}{(t_{4k-2}+3)(t_{4k-2}+2)(t_{4k-2}+1)f(t_{4k-2})} = \frac{k}{(t_{4k-2}+3)(t_{4k-2}+2)(t_{4k-2}+1)}\\
&\leq \frac{k}{%
(4k-2+3)(4k-2+2)(4k-2+1)} =\frac{1}{4(4k+1)(4k-1)}\rightarrow 0\;\emph{as}\;k\rightarrow \infty .
\end{align*}%
\emph{Hence, if we let $n_{k}:=t_{4k+1}$, for $k\in \mathbb{N}$, we have
proven that} 
\begin{equation}
\lim \frac{|\Delta \cap \lbrack 0,f(n_{k}))|}{f(n_{k})}=0.  \label{R1E3}
\end{equation}%
\emph{Since} $\frac{|\Delta \cap \lbrack 0,f(n_{k}))|}{f(n_{k})}$ \emph{is a
subsequence of }$\frac{|\Delta \cap \{1, \cdots, n\}|}{n}$, \emph{(\ref{R1E3})
above establishes that} $\liminf \;\frac{|\Delta \cap \{1, \cdots, n\}|}{n}=0$, 
\emph{as desired.} 

\emph{Next we sketch the proof that the lower asymptotic density of $\Delta(r)$ is zero.
To this end, observe that for the sequence $\{n_{k}(r):k\in\N\}$ used in the construction of $\langle z(r)\rangle$ the following holds from some $k\in \mathbb{N}$ onward: $n_{k+1}(r)>3+n_{k}(r)$ and $n_{k}(r)>k$.}
\emph{Therefore,} 
\begin{align*}
\frac{|\Delta(r)\cap [0, f(n_{k}(r)))|}{f(n_{k+1}(r)-1)}& \leq \frac{kf(n_{k}(r))}{f(n_{k}(r)+2)}=\frac{kf(n_{k}(r))}{(n_{k}(r)+2)(n_{k}(r)+1)f(n_{k}(r))}
\\
& =\frac{k}{(n_{k}(r)+2)(n_{k}(r)+1)}\leq \frac{k}{(k+2)(k+1)} \\
& \leq \frac{1}{k+2}\rightarrow 0\;\emph{as}\;k\rightarrow \infty .
\end{align*}%
\emph{A similar argument applies to $\Delta(rs)$ and $\Delta^{\prime}(rs)$.}
\end{remark}

\section{Concluding Remarks}
\label{sec:4}
We close this paper with a  remark on further research we plan to undertake in the future.
\citet{petri2019} proved that a social welfare order satisfying AN and LAP on an infinite domain Y admits no real-valued representation. 
This leaves open the question of whether such a social welfare order can be described explicitly. 
We are currently working on this open question in a companion paper.

\bibliographystyle{plainnat}
\bibliography{APAnonymity}

\end{document}